\documentclass[letterpaper,10pt,conference]{ieeeconf}
\IEEEoverridecommandlockouts
\usepackage{amsmath,amssymb,amsthm,mathtools,bm}
\usepackage{algorithm}
\usepackage{algpseudocode}
\usepackage{tabularx}
\usepackage{xcolor}
\usepackage{subfig}

\newtheorem{lemma}{Lemma}
\newtheorem{proposition}{Proposition}
\newtheorem{theorem}{Theorem}
\newtheorem{assumption}{Assumption}
\newtheorem{remark}{Remark}

\newcommand{\R}{\mathbb{R}}
\newcommand{\M}{\mathcal{M}}
\newcommand{\norm}[1]{\left\lVert #1\right\rVert}
\newcommand{\inner}[2]{\left\langle #1,#2\right\rangle}

\newcommand{\grad}{\operatorname{grad}}
\newcommand{\Exp}{\operatorname{Exp}}
\newcommand{\Log}{\operatorname{Log}}
\newcommand{\dist}{\operatorname{dist}}

\newcommand{\argmin}{\operatorname*{arg\,min}}
\newcommand{\D}{\mathrm{D}}

\title{\LARGE\bf
Riemannian Density-Driven Optimal Control:
Tangent-Space LQR for Second-Order Multi-Agent Systems on Curved Manifolds}

\author{
Kooktae Lee$^{1}$ and Ruchika Singh$^{1}$
\thanks{*This work was supported by NSF CAREER Grant CMMI-DCSD-2638508.}
\thanks{$^{1}$Kooktae Lee and Ruchika Singh are with the Department of Mechanical and Aerospace Engineering, Texas Tech University, Lubbock, TX 79409, USA, email: kooktae.lee@ttu.edu, ruchisin@ttu.edu}  
}

\begin{document}
\maketitle
\thispagestyle{empty}
\pagestyle{empty}

\begin{abstract}
Density-Driven Optimal Control (D$^2$OC) provides an effective framework for steering multi-agent systems toward prescribed spatial distributions. However, existing D$^2$OC formulations are primarily developed for Euclidean domains and do not directly account for intrinsic manifold geometry. This paper extends D$^2$OC to second-order multi-agent systems evolving on Riemannian manifolds. The proposed Riemannian D$^2$OC ($R$-D$^2$OC) constructs a local distribution objective in the tangent space of each agent through logarithmic maps and uses its weighted center as the reference for a finite-horizon LQR. The resulting control is executed on the manifold through intrinsic second-order dynamics and parallel transport within a receding-horizon scheme. We establish local curvature-dependent bounds that quantify the approximation introduced by the tangent-space reduction and characterize the resulting target bias. Furthermore, we derive a conditional discrete-descent result showing that the closed-loop objective decreases when a local velocity-alignment condition is satisfied. Numerical simulations on a 3D ellipsoidal manifold demonstrate distribution-level control and empirically support the proposed approximation and descent results.
\end{abstract}
\section{Introduction}

Steering multi-agent systems toward prescribed spatial distributions while minimizing control effort is a fundamental challenge in robotics \cite{seo2025d2oc}. Density-Driven Optimal Control (D$^2$OC) addresses this via a Lagrangian framework penalizing spatial distribution mismatch along control effort \cite{seo2025d2oc,seo2026farm,lee2025connectivity,lee2026tac,lee2026automatica}. Existing D$^2$OC methods, however, rely on Euclidean spaces where quadratic density objectives yield standard linear quadratic tracking representations. When agent configurations evolve on curved manifolds, these Euclidean formulations break down because linear vector operations violate intrinsic geometric constraints.

To address these geometric limitations in multi-agent control, prior efforts have broadly explored three avenues. First, geometric control on Riemannian manifolds and Lie groups intrinsically resolves rotational singularities for attitude and rigid-body tracking \cite{bullo2004geometric,lee2008geometric,lee2015global}. Yet, these methods address internal system configuration constraints rather than agent navigation over environmentally constrained curved surfaces or collective spatial distribution control. Second, geodesic path-planning methods leverage intrinsic distances to generate valid or shortest paths on curved spaces \cite{wu2016path,chen2014smooth}, while primarily focusing on geometric path generation for individual agents without accounting for system dynamics or collective distribution control. Third, swarm and coverage-control methods shape agent configurations via Voronoi or density-based objectives \cite{cortes2004coverage,bullo2009distributed,krishnan2022multiscale}, but are predominantly Euclidean and do not preserve intrinsic manifold geometry. Thus, an optimal distribution-control framework accounting for intrinsic geometry, system dynamics, and computational efficiency remains undeveloped.

To resolve these limitations, we develop a local Riemannian realization of D$^2$OC by constructing the distribution target in the tangent space at the current agent state. The logarithmic map converts local reference samples of target distribution into tangent vectors, whose weighted average defines a virtual target for a finite-horizon LQR. The resulting control is then executed through the intrinsic second-order dynamics using parallel transport, while the receding-horizon update continuously shifts the tangent-space reference with the agent state. This construction retains the fixed-dimensional structure of LQR while providing explicit control over the geometric approximation introduced by the tangent-space reduction. We establish a geometric alignment condition under which a local discrete-time descent property holds, providing a formal criterion for manifold-based tracking control.

The main contributions of this paper are thus threefold:
1) a local Riemannian formulation of D$^2$OC ($R$-D$^2$OC) for second-order agents on curved spaces;
2) a rigorous analysis of approximation errors introduced by the tangent-space reduction under local curvature; and
3) a conditional discrete-descent guarantee for the resulting receding-horizon controller under a local velocity-alignment condition.

\section{Preliminaries and Problem Formulation}
\subsection{Density-Driven Optimal Control Framework}

The Euclidean density-driven optimal control framework steers a multi-agent system toward a prescribed spatial distribution represented by weighted reference samples. As described in \cite{seo2025d2oc,seo2026farm,lee2025connectivity,lee2026automatica}, the process consists of three iterative stages.
For a given coverage mission, a reference distribution is represented by sample points $\{q_j\}_{j=1}^{N}$ with uniform weights $1/N$.
In Stage A, each agent identifies a local subset of reference samples and computes a control input driving its motion toward the corresponding weighted spatial target. Stage B decreases the sample weights based on coverage progress, while Stage C exchanges updated weights among neighboring agents within communication range. Repeated execution enables progressive coverage of the reference distribution. Due to space limits, further details are omitted and can be found in the cited literature.

This paper focuses on Stage A. At each time step, an agent selects a local subset of reference samples and constructs a local target represented by their weighted barycenter. The agent then solves a finite-horizon optimal control problem to track this target while minimizing control effort. In Euclidean $D^2OC$, this problem reduces to standard linear-quadratic control since agent dynamics and targets are represented in Euclidean coordinates.
However, this construction relies on linear vector operations that are not intrinsically defined on a general Riemannian manifold. Consequently, local target construction and its LQR realization cannot be directly applied to agents moving on curved spaces. In what follows, we develop an intrinsic counterpart of Stage A using logarithmic maps and tangent-space dynamics.

\subsection{Density Representation on Manifold}

To lift the aforementioned Euclidean $D^2OC$ framework onto generic curved domains, we now formulate the density representation and optimal transport metrics on a smooth Riemannian manifold. Consider agent $i$ with configuration $x_i(t)\in\mathcal{M}$, where $(\mathcal{M},g)$ is a smooth $d$-dimensional Riemannian manifold. Let $v_i(t)\in T_{x_i(t)}\mathcal{M}$ denote its velocity. At planning step $k$, the agent considers its local subset of reference points $\mathcal{S}_i^k$. The local desired density for agent $i$ is represented by the finite positive measure
$
    \nu_i^k
    =
    \sum_{j\in\mathcal S_i^k}\pi_j\delta_{q_j},
    \;
    q_j\in\M,\quad \pi_j\geq0,
    \label{eq:measure}
$
with total local mass defined as
$
\alpha_i
=
\sum_{j\in\mathcal S_i^k}\pi_j
=
\frac{1}{M_i}
>0,
\label{eq:mass}
$
where $\pi_j$ represents the mass transport weight associated with reference point $q_j$, $\delta_{(\cdot)}$ denotes a Dirac measure, and $M_i$ stands for the fixed total agent capacity assigned to agent $i$. The symbol $M_i$ is determined by the ratio of total operation time to the sampling interval ($T_{\mathrm{total}}/\Delta t$), reflecting a uniform per-step energy consumption profile.

Extending the Wasserstein distance concept to the Riemannian setting, the optimal transport cost between the agent's current state (modeled as a Dirac measure at $x_i$ with mass $\alpha_i$) and the local reference measure $\nu_i^k$ is given by the squared Riemannian Wasserstein distance:
\begin{equation}
\textstyle
    \mathcal{W}_2^2(\alpha_i \delta_{x_i}, \nu_i^k)
    =
    \sum_{j\in\mathcal S_i^k}
    \pi_j\dist_g(x_i,q_j)^2,
    \label{eq:wasserstein}
\end{equation}
where $\dist_g(x_i, q_j)$ denotes the Riemannian distance between $x_i$ and $q_j$ induced by the metric tensor $g$ on the configuration manifold $M$.
Thus, the density mismatch naturally translates into a weighted sum of squared Riemannian distances over the agent's local sample set. 

To characterize the geometric center toward which the agent is steered, the corresponding local Riemannian barycenter for agent $i$ is defined as the minimizer of this transport cost:
\vspace{-.05in}
\begin{equation}
    \bar q_{R,i}
    \in
    \textstyle\argmin_{z\in\mathcal{U}}
    \frac12
    \sum_{j\in\mathcal S_i^k}
    \pi_j\dist_g(z,q_j)^2,
    \label{eq:barycenter}
\end{equation}
where $\mathcal{U}$ is a strongly convex neighborhood containing the relevant local sample points.

\subsection{Intrinsic Second-Order Dynamics}
For notational simplicity, the agent index $i$ is omitted hereafter.
Each agent obeys
\begin{equation}
    \dot x=v,
    \qquad
    \nabla_{\dot x}v=u,
    \qquad
    u(t)\in T_{x(t)}\M,
    \label{eq:dynamics}
\end{equation}
where $\nabla$ is the Levi--Civita connection, thus $u$ is the covariant
acceleration. We consider the Bolza problem
\begin{equation}
\small
\begin{aligned}
J
&=
\int_{t_0}^{t_f}
\left[
V(x(t))
+
\frac12\norm{u(t)}_{x(t)}^2
\right]dt
+
\Phi(x(t_f)),
\\
V(x)
&=
\frac12
\textstyle\sum_{j\in\mathcal{S}^k}
\pi_j\dist_g(x,q_j)^2,
\\
\Phi(x)
&=
\frac{w_f}{2}
\textstyle\sum_{j\in\mathcal{S}^k}
\pi_j\dist_g(x,q_j)^2,
\qquad w_f\geq0,
\end{aligned}
\label{eq:ocp}
\end{equation}
subject to \eqref{eq:dynamics} and a fixed initial state. The constant $w_f \ge 0$ penalizes the terminal distribution error relative to the running effort.

\section{Tangent-Space Reduction}

\subsection{Normal Coordinates}

To map global geometric entities from the Riemannian manifold $\mathcal{M}$ into a local Euclidean vector space suitable for linear control design, we utilize exponential and logarithmic maps. The exponential map $\Exp_{x_k} \colon T_{x_k}\mathcal M \to \mathcal M$ maps a tangent vector from the tangent space at $x_k$ to a point along the geodesic on $\mathcal M$. Its inverse, the logarithmic map $\Log_{x_k} \colon \mathcal M \to T_{x_k}\mathcal M$, projects points from the manifold back onto the tangent space $T_{x_k}\mathcal M$, where the norm of the tangent vector corresponds to the Riemannian distance along the minimizing geodesic.

At replanning step $k$, we define the local coordinates in the tangent space via the logarithmic map as
\begin{equation}
    \xi
    =
    \Log_{x_k}(x),
    \qquad
    \eta_j
    =
    \Log_{x_k}(q_j).
    \label{eq:normal_coordinates}
\end{equation}
By choosing an orthonormal basis of $T_{x_k}\mathcal M$, we explicitly identify this tangent space with $\R^d$.

For each $x$ in the local normal neighborhood, let
$\mathcal P_{x\rightarrow x_k}$ denote parallel transport along the
unique radial geodesic
\[
s\mapsto \Exp_{x_k}(s\xi),
\qquad s\in[0,1].
\]
Define the transported velocity
$
    \zeta
    =
    \mathcal P_{x\rightarrow x_k}v
    \in T_{x_k}\M.
\label{eq:transported_velocity}
$

Assume the sectional curvature satisfies $|\operatorname{Sec}|\leq\kappa$
on a compact normal neighborhood containing the local trajectory and
the relevant target samples. To characterize the linearization error on the manifold, consider the radial geodesic $\gamma(s) = \Exp_{x_k}(s\xi)$ for $s \in [0,1]$. For any direction $v \in T_{x_k}M$, the action of the differential $\D\Exp_{x_k}(\xi)v$ is represented by a Jacobi field $J(s)$ along $\gamma(s)$ with $J(0)=0$ and $\nabla_{\dot{\gamma}}J(0)=v$. Integrating the Jacobi equation $\nabla_{\dot{\gamma}}^2 J + R(J, \dot{\gamma})\dot{\gamma} = 0$ yields the integral form
\begin{equation}
J(1) = v - \int_0^1 (1-s) R(J(s), \dot{\gamma})\dot{\gamma} \, ds.
\end{equation}
Noting that $J(1) = \D\Exp_{x_k}(\xi)v$, applying parallel transport $\mathcal{P}_{x\rightarrow x_k}$ along $\gamma$ and subtracting $v = Iv$ gives
\begin{equation}
    \begin{aligned}
    &\left( \mathcal{P}_{x\rightarrow x_k}\D\Exp_{x_k}(\xi) - I \right)v = \\
    &\qquad-\int_0^1 (1-s) \mathcal{P}_{\gamma(s)\rightarrow x_k} R(J(s), \dot{\gamma})\dot{\gamma} \, ds.
    \end{aligned}
\end{equation}

Taking the norm on both sides and using the sectional curvature bound $|\operatorname{Sec}|\leq\kappa$, the geodesic speed $\|\dot{\gamma}\| = \|\xi\|$, and the Jacobi field estimate $\|J(s)\| \le C_0 \|v\|$ (where $C_0 > 0$ is a local constant bounded near unity), the integral term is bounded by $C_1 \kappa \|\xi\|^2 \|v\|$ for a constant $C_1 > 0$ depending on $C_0$. Taking the supremum over $\|v\| = 1$ leads directly to the operator norm estimate
\begin{equation}
    \norm{
    \mathcal P_{x\rightarrow x_k}
    \D\Exp_{x_k}(\xi)
    -
    I
    }
    \le
    C_1\kappa\norm{\xi}^2.
    \label{eq:dexp_bound}
\end{equation}
To understand the dynamic behavior in the tangent space $T_{x_k}M$, recall that $\xi = \Log_{x_k}(x)$ and its velocity variable in the tangent space is defined via parallel transport as $\zeta = \mathcal{P}_{x \rightarrow x_k} \dot{x}$. Differentiating $\xi$ with respect to time yields $\dot{\xi} = \D\Log_{x_k}(x) \dot{x} = \D\Log_{x_k}(x) \mathcal{P}_{x_k \rightarrow x} \zeta$. Since the logarithmic map is the inverse of the exponential map, its differential satisfies $\D\Log_{x_k}(x) = (\D\Exp_{x_k}(\xi))^{-1}$. Applying the operator norm estimate from \eqref{eq:dexp_bound} to this inverse differential and utilizing Neumann series expansion, the leading-order linear mapping reduces to the identity operator $I$, while the curvature-induced perturbation is collected in the residual term $r_\xi$. Consequently, after possibly shrinking the neighborhood,
\begin{equation}
    \dot\xi
    =
    \zeta+r_\xi,
    \label{eq:xi}
\end{equation}
with
$
    \norm{r_\xi}
    \le
    C_2\kappa
    \norm{\xi}^2
    \norm{\zeta}.
    \label{eq:xi_bound}
$

Similarly, consider the acceleration dynamics under the intrinsic covariant derivative $\nabla_{\dot{x}} \dot{x} = u$. Differentiating the parallel-transported velocity $\zeta = \mathcal P_{x\rightarrow x_k} \dot{x}$ along the state trajectory $x(t)$ requires accounting for the curvature-induced non-commutativity of parallel transport along distinct paths (the path along the trajectory versus the radial geodesic connecting $x(t)$ to $x_k$). Evaluating this geometric drift via standard Jacobi-field and curvature tensor estimates along the geodesic triangle formed by $x_k$, $x(t)$, and $x(t+\Delta t)$ yields
\begin{equation}
    \dot\zeta
    =
    u_0+r_\zeta,
    \label{eq:zeta}
\end{equation}
where $u_0=\mathcal P_{x\rightarrow x_k}u$ represents the parallel-transported control input, and the residual acceleration error $r_\zeta$ is bounded by
$
    \norm{r_\zeta}
    \le
    C_3
    \left(
    \kappa\norm{\xi}\norm{\zeta}^2
    +
    \kappa\norm{\xi}^2\norm{u_0}
    \right).
    \label{eq:zeta_bound}
$

The constants $C_1,C_2,C_3$ depend only on the selected compact neighborhood and uniform bounds on the relevant geometric quantities, confirming that the Euclidean double integrator $(\dot{\xi} = \zeta, \dot{\zeta} = u_0)$ serves as the exact leading-order local model in the fixed tangent space $T_{x_k}M$.

\subsection{Local Expansion of the Density Objective}

For $\xi,\eta_j$ in a sufficiently small common normal neighborhood,
the squared Riemannian distance admits the local expansion
\begin{equation}
\small
\begin{aligned}
\frac12
\dist_g
\left(
\Exp_{x_k}(\xi),
\Exp_{x_k}(\eta_j)
\right)^2
&=
\frac12
\norm{\xi-\eta_j}_{x_k}^2
+
R_j(\xi,\eta_j).
\end{aligned}
\label{eq:distance_expansion}
\end{equation}
In normal coordinates centered at $x_k$, the squared Riemannian distance has a Euclidean quadratic leading term and a fourth-order curvature correction. Thus, for $\xi$ and $\eta_j$ in a sufficiently small common normal neighborhood, there exists $C_4>0$, depending only on the neighborhood, such that
\begin{equation}
    |R_j(\xi,\eta_j)|
    \le
    C_4\kappa
    \left(
    \norm{\xi}
    +
    \norm{\eta_j}
    \right)^4,
\label{eq:distance_remainder}
\end{equation}
with $C_4$ depending only on the chosen neighborhood. The remainder
vanishes identically when the neighborhood is flat. Therefore,
\begin{equation}
\small
V(x)
=
\frac12
\sum_{j\in\mathcal{S}^k}
\pi_j
\norm{\xi-\eta_j}^2
+
R_V(\xi),
\label{eq:potential_expansion}
\end{equation}
\vspace{-.2in}
\begin{flalign}
&\text{where }
    |R_V(\xi)|
    \le
    C_4\kappa
    \sum_{j\in\mathcal{S}^k}
    \pi_j
    \left(
    \norm{\xi}
    +
    \norm{\eta_j}
    \right)^4.&
\label{eq:potential_remainder}
\end{flalign}

Defining the weighted tangent-space center
$\eta_{\mathrm{lin}}=\frac1{\alpha}\sum_j\pi_j\eta_j$, we have
$
\frac12
\sum_j
\pi_j
\norm{\xi-\eta_j}^2
=
\frac{\alpha}{2}
\norm{\xi-\eta_{\mathrm{lin}}}^2
+
c_k,
\label{eq:complete_square}
$
where $c_k=\frac12\sum_j\pi_j\norm{\eta_j-\eta_{\mathrm{lin}}}^2$ is
independent of $\xi$. The resulting tangent-space virtual target is $\eta_{\mathrm{lin}}$, which is mapped back to the manifold as
\begin{equation}
\tilde{q}_R
=
\operatorname{Exp}_{x_k}(\eta_{\mathrm{lin}}).
\label{eq:virtual_target}
\end{equation}

\begin{remark}
Note that $\eta_{\mathrm{lin}}$ serves as a tangent-space virtual target, which coincides with the true Riemannian barycenter in flat space but differs under nonzero curvature.
\end{remark}
\section{Tangent-Space Target Aggregation and LQR Control}
\label{sec:local_lqr}
\subsection{Local LQR Controller}

Neglecting the local geometric remainders in
\eqref{eq:xi}--\eqref{eq:potential_remainder} gives the tangent-space
approximation
\begin{equation}
\begin{aligned}
\min_{u_0(\cdot)}
\quad
J_{\mathrm{loc}}
&=
\frac12
\int_{t_k}^{t_f}
\left[
\alpha
\norm{\xi-\eta_{\mathrm{lin}}}^2
+
\norm{u_0}^2
\right]dt
\\
&\quad+
\frac{w_f\alpha}{2}
\norm{\xi(t_f)-\eta_{\mathrm{lin}}}^2,
\\
\text{s.t.}\qquad \dot\xi&=\zeta,
\quad
\dot\zeta=u_0,
\end{aligned}
\label{eq:local_lqr_cont}
\end{equation}
where the terminal weight follows directly from \eqref{eq:ocp}. Define the tracking error $e=\xi-\eta_{\mathrm{lin}}$ and state vector $y=[e^\top,\zeta^\top]^{\top}$. Since $\eta_{\mathrm{lin}}$ is fixed during one replanning horizon, the resulting nominal linear error dynamics satisfy
$
    \dot y
    =
    Ay+Bu_0,
\label{eq:continuous_linear}
$
where
$
\tiny
A=
\begin{bmatrix}
0&I\\
0&0
\end{bmatrix},
\,
B=
\begin{bmatrix}
0\\
I
\end{bmatrix}.
\label{eq:AB}
$
An exact zero-order-hold discretization with step $\Delta t$ gives
$y_{m+1}=A_dy_m+B_du_m$  with
$
\tiny
A_d=
\begin{bmatrix}
I&\Delta t I\\
0&I
\end{bmatrix},
\,
B_d=
\begin{bmatrix}
\frac12\Delta t^2I\\
\Delta t I
\end{bmatrix}
\label{eq:AdBd}
$
at discrete time index $m$.
Approximating the running integral gives
\begin{equation}
J_d
=
\frac12
y_T^\top Q_fy_T
+
\sum_{m=0}^{T-1}
\frac12
\left(
y_m^\top\bar Qy_m
+
u_m^\top\bar Ru_m
\right),
\label{eq:discrete_cost}
\end{equation}
with
$\bar Q=\Delta t\,\mathrm{diag}(\alpha I,0)$,
$\bar R=\Delta t I$, and
$Q_f=\mathrm{diag}(w_f\alpha I,0)$, thus $P_T=Q_f$.
Backward dynamic programming gives $u_m^*=-K_my_m$ with
\begin{equation}
\begin{aligned}
    S_m
    &=
    \bar R
    +
    B_d^\top P_{m+1}B_d,
    \qquad
    K_m
    =
    S_m^{-1}
    B_d^\top P_{m+1}A_d,\\
P_m
&=
\bar Q
+
A_d^\top P_{m+1}A_d
-
A_d^\top P_{m+1}B_d
S_m^{-1}
B_d^\top P_{m+1}A_d .
\end{aligned}
\label{eq:riccati}
\end{equation}
Because $\bar R\succ0$ and $P_{m+1}\succeq0$, $S_m\succ0$, the backward recursion is well-defined.

\begin{remark}[Fixed-Dimensional Computational Complexity]
\label{rem:complexity}
For a horizon $T$, solving the discrete Riccati recursion requires $\mathcal{O}(Td^3)$ operations, as the state dimension of $y$ depends solely on the manifold dimension $d$. Meanwhile, computing the tangent-space target $\eta_{\mathrm{lin}}$ via weighted summation requires $\mathcal{O}(|\mathcal{S}^k|d)$ operations. Consequently, the target sample size $|\mathcal{S}^k|$ affects only the linear vector aggregation, leaving the matrix Riccati state dimension completely invariant to sample size.
\end{remark}
\subsection{Receding-Horizon Implementation for $R$-D$^2$OC}

At replanning step $k$, the local control command is synthesized in the tangent space $T_{x_k}\mathcal{M}$ and updated recursively. Given the current manifold state $x_k \in \mathcal{M}$ and a set of discrete target points $\{q_j\}$, we first compute the logarithmic map coordinates $\eta_j^k = \mathrm{Log}_{x_k}(q_j) \in T_{x_k}\mathcal{M}$. The weighted spatial centroid in $T_{x_k}\mathcal{M}$ is evaluated as $\eta_{\mathrm{lin}}^k = \frac{1}{\alpha}\sum_{j}\pi_j\eta_j^k$ with total weight $\alpha = \sum_{j}\pi_j$. 

Since $x_k$ serves as the origin of the chosen orthonormal basis for $T_{x_k}\mathcal{M}$, the initial state vector for the local LQR policy is formulated as $y_k = [(-\eta_{\mathrm{lin}}^k)^{\top}, v_k^{\top}]^{\top}$, where $v_k \in T_{x_k}\mathcal{M}$ is the current velocity vector. The feedback acceleration command is then obtained as $u_k = -K_0 y_k \in T_{x_k}\mathcal{M}$. 

To execute this control over the interval $t \in [t_k, t_k+\Delta t]$, $u_k$ is continuously pushed forward along the trajectory $x(t)$ via parallel transport, $u(t) = \mathcal{P}_{x_k \rightarrow x(t)} u_k$, driving the intrinsic system dynamics $\dot x = v$ and $\nabla_{\dot x}v = u(t)$. A local second-order numerical update over the time step $\Delta t$ is then used to advance the state on $\mathcal{M}$:
\begin{equation}
\begin{aligned}
x_{k+1} = \mathrm{Exp}_{x_k}\left( \Delta t v_k + \frac{1}{2}\Delta t^2 u_k \right), \\
v_{k+1} = \mathcal{P}_{x_k \rightarrow x_{k+1}}\left( v_k + \Delta t u_k \right).
\end{aligned}\label{eq:position_update}
\end{equation}
This update provides a local approximation of the intrinsic dynamics over each sampling interval. The tangent-space origin is then shifted from $x_k$ to $x_{k+1}$ at each step to complete the receding-horizon control loop.

\section{Theoretical Analysis: Approximation Bias and Local Conditional Descent}
\label{sec:theoretical_analysis}

\subsection{Objective Approximation Error}
 
Define the local objective approximation error at an arbitrary
trajectory point
\begin{equation}
\small
E(x;x_k)
=
V(x)
-
\frac12
\textstyle\sum_j
\pi_j
\norm{
\Log_{x_k}(x)-\eta_j^k
}_{x_k}^2,
\label{eq:Ek_definition}
\end{equation}
where $\eta_j^k=\Log_{x_k}(q_j)$. This is exactly the gap between the
true running cost and the quadratic surrogate the local LQR controller
actually minimizes. 
 
\begin{proposition}[Local objective error]
\label{prop:objective_error}
Suppose $x$ and all target samples belong to a common sufficiently
small normal neighborhood of $x_k$. Then
\begin{equation}
|E(x;x_k)|
\le
C_4\kappa\alpha
\left(
\norm{\xi}_{x_k}
+
\sigma_k^{\mathrm{tar}}
\right)^4,
\label{eq:objective_error_bound}
\end{equation}
where $\xi=\Log_{x_k}(x)$ for the local
coordinate of $x$ and $\sigma_k^{\mathrm{tar}}=\max_j\norm{\eta_j^k}_{x_k}$.
\end{proposition}
 
\begin{proof}
From \eqref{eq:potential_remainder},
$
|E(x;x_k)|
\le
C_4\kappa
\sum_j
\pi_j
\left(
\norm{\xi}
+
\norm{\eta_j^k}
\right)^4$
$\le
C_4\kappa\alpha
\left(
\norm{\xi}
+
\sigma_k^{\mathrm{tar}}
\right)^4.
$
\end{proof}
 
This is a pointwise approximation, not a cumulative convergence result. A cumulative bound requires additional assumptions controlling target spread and prediction error along a trajectory. 

\subsection{Bias of the Tangent-Space Center}

To characterize the Riemannian barycenter $\bar q_R$ in \eqref{eq:barycenter}, define the weighted Fréchet energy \cite{absil2008optimization} as
$\mathcal F(x)=\frac12\sum_j\pi_j\dist_g(x,q_j)^2$,
whose critical point satisfies $\grad\mathcal F(\bar q_R)=0$.
Although $\mathcal F$ is identical to the state cost $V$ in \eqref{eq:ocp}, we use $\mathcal F$ here to emphasize its role in the barycenter analysis. To quantitatively compare $\tilde q_R$ in \eqref{eq:virtual_target} with $\bar q_R$, we require $\mathcal F$ to be locally strongly convex near the barycenter. Otherwise, their intrinsic distance cannot be stably controlled.

\begin{assumption}[Nondegenerate local barycenter]
\label{ass:barycenter}
There exists a strongly convex neighborhood $\mathcal U_B$ containing
the reference point $x_0$, all target samples $q_j$, the tangent-space
virtual target $\tilde q_R$, and the Riemannian barycenter $\bar q_R$,
such that the minimizing geodesics used below remain in $\mathcal U_B$
and
\begin{equation}
    \operatorname{Hess}\mathcal F(x)
    \succeq
    \lambda_{\min}I,
    \qquad
    x\in\mathcal U_B,
\label{eq:hess_lower}
\end{equation}
for some $\lambda_{\min}>0$.
\end{assumption}
\begin{remark}
Assumption~\ref{ass:barycenter} is standard in Riemannian optimization \cite{absil2008optimization}. As $\mathcal F$ is a weighted sum of squared distance functions, $\operatorname{Hess}\mathcal F(x) = \alpha I + \mathcal O(\kappa \alpha \rho_{\max}^2)$ with $\alpha = \sum_j \pi_j$. Thus, strong convexity naturally holds within a sufficiently small normal neighborhood.
\end{remark}

\begin{lemma}[Local change-of-base-point estimate]
\label{lem:log_change_base}
Let $x_0$ and $q$ belong to a sufficiently small strongly convex
normal neighborhood, $\eta=\Log_{x_0}(q)$, $\delta\in T_{x_0}\M$,
$\tilde x=\Exp_{x_0}(\delta)$. Assume $|\operatorname{Sec}|\le\kappa$
on the neighborhood. Then, there exists $C_{5}>0$, independent of
$\norm{\eta}+\norm{\delta}$, such that
\begin{equation}
\begin{aligned}
\mathcal P_{\tilde x\rightarrow x_0}
\Log_{\tilde x}(q)
&=
\eta-\delta+r(\eta,\delta),
\\
\norm{r(\eta,\delta)}_{x_0}
&\le
C_{5}\kappa
\left(
\norm{\eta}
+
\norm{\delta}
\right)^3.
\label{eq:log_change_base}
\end{aligned}
\end{equation}
\end{lemma}

\begin{proof}
This is the local third-order change-of-base-point expansion for the
Riemannian logarithmic map, following from the Jacobi-field expansion
of the exponential map and the variation of radial parallel transport.
The Euclidean relation $\eta-\delta$ is exact when curvature vanishes.
The first nonzero curvature correction is cubic in the local
displacement, with uniform remainder on a compact strongly convex
neighborhood.
\end{proof}

This local expansion serves as a key technical tool in our framework, enabling us to compare $\operatorname{grad}\mathcal{F}$ evaluated at two distinct base points within the same vector space, up to a controlled cubic error.

\begin{theorem}[Local curvature bias]
\label{thm:bias}
Let $\rho_{\max}=\max_j\norm{\Log_{x_0}(q_j)}_{x_0}$,
$\eta_{\mathrm{lin}}=\frac{1}{\alpha}\sum_j\pi_j\Log_{x_0}(q_j)$,
$\tilde q_R=\Exp_{x_0}(\eta_{\mathrm{lin}})$. Suppose the conditions of
Assumption~\ref{ass:barycenter} hold with some $\lambda_{min}>0$, $|\operatorname{Sec}|\le\kappa$
on $\mathcal U_B$, and $\rho_{\max}\le\rho_*$ for a sufficiently small
$\rho_*>0$. Then, there exists $C_6>0$, independent of $\rho_{\max}$,
such that
\begin{equation}
\dist_g(\tilde q_R,\bar q_R)
\le
C_6
\frac{\kappa\alpha}{\lambda_{\min}}
\rho_{\max}^3.
\label{eq:bias}
\end{equation}
\end{theorem}

\begin{proof}
Let $\eta_j=\Log_{x_0}(q_j)$. By definition,
$\norm{\eta_j}\le\rho_{\max}$ and $\norm{\eta_{\mathrm{lin}}}\le\rho_{\max}$.
Applying Lemma~\ref{lem:log_change_base} with $\delta=\eta_{\mathrm{lin}}$,
$\tilde x=\tilde q_R$, for each $j$, we have
\begin{equation}
\mathcal P_{\tilde q_R\rightarrow x_0}
\Log_{\tilde q_R}(q_j)
=
\eta_j-\eta_{\mathrm{lin}}+r_j,
\;
\norm{r_j}_{x_0}
\le
8C_{5}\kappa\rho_{\max}^3.
\label{eq:log_expansion}
\end{equation}
By the same first-variation formula \cite{absil2008optimization} applied at
$\tilde q_R$, $\grad\mathcal F(\tilde q_R)=
-\sum_j\pi_j\Log_{\tilde q_R}(q_j)$. Parallel transporting to
$T_{x_0}\M$ and using \eqref{eq:log_expansion},
$
\mathcal P_{\tilde q_R\rightarrow x_0}
\grad\mathcal F(\tilde q_R)
=
-\sum_j
\pi_j
\left(
\eta_j-\eta_{\mathrm{lin}}+r_j
\right)
=
-\sum_j
\pi_jr_j,
$
due to the fact $\sum_j\pi_j(\eta_j-\eta_{\mathrm{lin}})=0$. 
Since the parallel transport preserves the Riemannian norm, we have
\begin{equation}
    \|\operatorname{grad}\mathcal{F}(\tilde{q}_R)\|_{\tilde{q}_R}
    \le
    8C_5 \kappa \alpha \rho_{\max}^3.
\label{eq:gradient_bias}
\end{equation}

Let $\gamma:[0,1]\rightarrow\M$ be the minimizing geodesic from
$\bar q_R$ to $\tilde q_R$ with constant speed, remaining in
$\mathcal U_B$ by Assumption~\ref{ass:barycenter}. Since
$\grad\mathcal F(\bar q_R)=0$, parallel transport along $\gamma$ and
the fundamental theorem of calculus give
\begin{equation}
\mathcal P_{\tilde q_R\rightarrow\bar q_R}
\grad\mathcal F(\tilde q_R)
=
\int_0^1
\mathcal P_{\gamma(s)\rightarrow\bar q_R}
\left[
\operatorname{Hess}\mathcal F(\gamma(s))
\dot\gamma(s)
\right]ds .
\end{equation}
Because $\dot\gamma(s)$ is parallel along the geodesic,
Assumption~\ref{ass:barycenter} gives
\vspace{-.15in}
\begin{equation}
    \begin{aligned}
\inner{
\mathcal P_{\tilde q_R\rightarrow\bar q_R}
\grad\mathcal F(\tilde q_R)
}{
\dot\gamma(0)
}_{\bar q_R}
&\ge
\lambda_{\min}
\int_0^1
\norm{\dot\gamma(s)}^2ds\\
&=
\lambda_{\min}
\dist_g(\bar q_R,\tilde q_R)^2.
    \end{aligned}
\end{equation}

By Cauchy--Schwarz,
$\lambda_{\min}\dist_g(\bar q_R,\tilde q_R)\le\norm{\grad\mathcal
F(\tilde q_R)}$, which combined with \eqref{eq:gradient_bias} gives
\eqref{eq:bias} with $C_6=8C_{5}$.
\end{proof}

\subsection{Local Descent under Tangent-Space LQR}

In the receding-horizon implementation, the target
$\bar q$ is instantiated at each replanning step as the tangent-space
target $\tilde q_R^k$. The following result directly connects the
proposed LQR controller to local descent. In particular, the required
alignment condition is imposed on the velocity generated by the
tangent-space LQR and explicitly accounts for the local target spread.

For a fixed target $\bar q\in\M$, define
\begin{equation}
    \Psi(x)
    =
    \frac12\dist_g(x,\bar q)^2,
    \qquad
    r(x)=\Log_x(\bar q).
\label{eq:dW}
\end{equation}
The first variation formula \cite{absil2008optimization} gives
\begin{equation}
    d\Psi_x[v]
    =
    -\inner{r(x)}{v}_x.
\end{equation}

\begin{assumption}[LQR alignment and local boundedness]
\label{ass:lqr_alignment}
For the tangent-space LQR controller in
Section~\ref{sec:local_lqr}, there exist constants
$c_1,c_2,\Gamma,L_v,L_u>0$ and a sufficiently small neighborhood such
that, at each replanning step,
\begin{equation}
    \inner{r_k}{v_k}_{x_k}
    \ge
    c_1\norm{r_k}_{x_k}^2
    -
    c_2\kappa\rho_k^3,
\label{eq:lqr_alignment}
\end{equation}
where
$
    r_k=\Log_{x_k}(\bar q),
    \;
    \rho_k
    =
    \max\left\{
    \norm{r_k}_{x_k},
    \sigma_k^{\mathrm{tar}}
    \right\},
    \;
    \sigma_k^{\mathrm{tar}}
    =
    \max_j\norm{\eta_j^k}_{x_k},
$
and
\begin{equation}
    \sigma_k^{\mathrm{tar}}
    \le
    \Gamma\norm{r_k}_{x_k}.
\label{eq:target_spread_alignment}
\end{equation}
Moreover, along the corresponding LQR closed-loop segment
$t\in[t_k,t_k+\Delta t]$,
\begin{equation}
    \norm{v(t)}_{x(t)}
    \le L_v\rho_k,
    \qquad
    \norm{u(t)}_{x(t)}
    \le L_u\rho_k.
\label{eq:lqr_local_bounds}
\end{equation}
The constants are uniform over the selected local neighborhood.
\end{assumption}

\begin{theorem}[Local descent under tangent-space LQR]
\label{thm:descent}
Suppose the trajectory generated by the tangent-space LQR remains in
a compact strongly convex neighborhood on which $\Psi$ has uniformly
bounded Hessian. Under Assumption~\ref{ass:lqr_alignment}, there exist
$\rho^*>0$, $\Delta t^*>0$, and $C_7>0$ such that, whenever
$\rho_k\le\rho^*$ and $0<\Delta t\le\Delta t^*$,
\begin{equation}
    \Psi(x_{k+1})-\Psi(x_k)
    \le
    -C_7\Delta t
    \norm{r_k}_{x_k}^2.
\label{eq:descent}
\end{equation}
Thus, the proposed tangent-space LQR yields local descent whenever
its closed-loop velocity satisfies the gain-dependent alignment
condition \eqref{eq:lqr_alignment}.
\end{theorem}

\begin{proof}
Let
$\rho_k=\max\{\norm{r_k}_{x_k},\sigma_k^{\mathrm{tar}}\}$.
By \eqref{eq:target_spread_alignment},
\begin{equation}
    \rho_k
    \le
    \Gamma_\rho\norm{r_k}_{x_k},
    \qquad
    \Gamma_\rho:=\max\{1,\Gamma\}.
\label{eq:rho_relation}
\end{equation}
Along the LQR closed-loop trajectory,
\begin{equation}
\small
    \frac{d^2}{dt^2}\Psi(x(t))
    =
    \operatorname{Hess}\Psi(x(t))[v(t),v(t)]
    +
    \inner{\grad\Psi(x(t))}{u(t)}_{x(t)}.
\end{equation}
Since the trajectory remains in a compact strongly convex neighborhood,
there exists $C_H>0$ such that
$\norm{\operatorname{Hess}\Psi}\le C_H$. Moreover,
$\norm{\grad\Psi(x(t))}=\dist_g(x(t),\bar q)$ is locally bounded by
$C_r\rho_k$ for some uniform $C_r>0$. Using
\eqref{eq:lqr_local_bounds} therefore gives
\begin{equation}
\left|
\frac{d^2}{dt^2}\Psi(x(t))
\right|
\le
\left(
C_HL_v^2+C_rL_u
\right)\rho_k^2
=:C_8\rho_k^2.
\label{eq:second_derivative_bound}
\end{equation}
Taylor's formula consequently yields
\begin{equation}
\begin{aligned}
\Psi(x_{k+1})-\Psi(x_k)
&\le
-\Delta t\inner{r_k}{v_k}_{x_k}
+
\frac12C_8\Delta t^2\rho_k^2.
\end{aligned}
\label{eq:taylor_descent}
\end{equation}

Using \eqref{eq:lqr_alignment} and \eqref{eq:rho_relation},
\begin{equation}
\small
\begin{aligned}
\Psi(x_{k+1})-\Psi(x_k)
&\le
-c_1\Delta t\norm{r_k}_{x_k}^2
+
c_2\kappa\Delta t\rho_k^3
+
\frac12C_8\Delta t^2\rho_k^2
\\
&\le
\left[
-c_1
+
c_2\kappa\Gamma_\rho^3\rho_k
+
\frac12C_8\Gamma_\rho^2\Delta t
\right]
\Delta t
\norm{r_k}_{x_k}^2.
\end{aligned}
\label{eq:descent_prebound}
\end{equation}
Choose $\rho^*>0$ and $\Delta t^*>0$ such that
\begin{equation}
    c_2\kappa\Gamma_\rho^3\rho^*
    \le \frac{c_1}{4},
    \qquad
    \frac12C_8\Gamma_\rho^2\Delta t^*
    \le \frac{c_1}{4}.
\label{eq:smallness_conditions}
\end{equation}
Then, for $\rho_k\le\rho^*$ and
$0<\Delta t\le\Delta t^*$,
\begin{equation}
    \Psi(x_{k+1})-\Psi(x_k)
    \le
    -\frac{c_1}{2}\Delta t
    \norm{r_k}_{x_k}^2.
\end{equation}
Hence \eqref{eq:descent} holds with
$C_7=c_1/2$.
\end{proof}

\begin{remark}
The alignment condition
\eqref{eq:lqr_alignment} is imposed on the velocity generated by the
proposed tangent-space LQR, while \eqref{eq:lqr_local_bounds} captures
the corresponding local closed-loop bounds. The additional inequality \eqref{eq:target_spread_alignment} ensures that the target-spread term introduced by the tangent-space approximation is of the same local scale as the tracking error. Therefore, the result provides a direct
sufficient condition under which the proposed LQR inherits local
descent on the nonlinear manifold.
\end{remark}

\section{Simulation Results}
\label{sec:simulation}

To validate the theoretical guarantees, we implement the proposed $R$-D$^2$OC framework on a 3D ellipsoidal manifold, where multiple second-order agents coordinate to match a target distribution over the curved geometry.

Fig.~\ref{fig:simulation_trajectory} depicts the overall system behavior and spatial trajectories.
The agents successfully track the reference distribution along the Riemannian manifold while maintaining smooth coordination paths. Originating from tightly clustered initial positions (black crosses), trajectories evolve smoothly across the surface to match their time-averaged distribution with the targets (blue dots). This confirms that the second-order tangent-space LQR formulation respects geometric constraints while driving the team toward the desired distribution.

To evaluate the geometric approximation properties and empirically assess Theorem~\ref{thm:bias}, we analyze the log-log scaling relationship in Fig.~\ref{fig:descent_analysis}(a). The empirical approximation bias exhibits a fitted slope of $2.73$, which is consistent with the theoretical cubic-order $\mathcal{O}(\rho_{\max}^3)$ upper bound. This observation supports the predicted local scaling of the tangent-space center bias as the target spread varies.

We next evaluate the local descent property in Theorem~\ref{thm:descent} using Fig.~\ref{fig:descent_analysis}(b). The closed-loop trajectories satisfy the alignment condition of Assumption~\ref{ass:lqr_alignment} with $c_1 = 1.40$, yielding $C_7 = c_1/2 = 0.70$ in~\eqref{eq:descent}. Within the empirical local neighborhood $\rho_k \le \rho^* \approx 4.1$ (shaded region), the observed objective decrease aligns well with the theoretical bound $-C_7\Delta t\|r_k\|^2$. Beyond $\rho^*$, this local bound no longer holds, which is consistent with the conditional and local nature of Theorem~\ref{thm:descent}.

With an average per-agent Riccati solve time of $0.49~\text{ms}$ relative to the $0.1~\text{s}$ replanning interval $\Delta t$, the computation remains well within the budget for onboard execution.

\begin{figure}[!t]
\vspace{-0.2in}
    \centering
    \includegraphics[width=.82\linewidth]{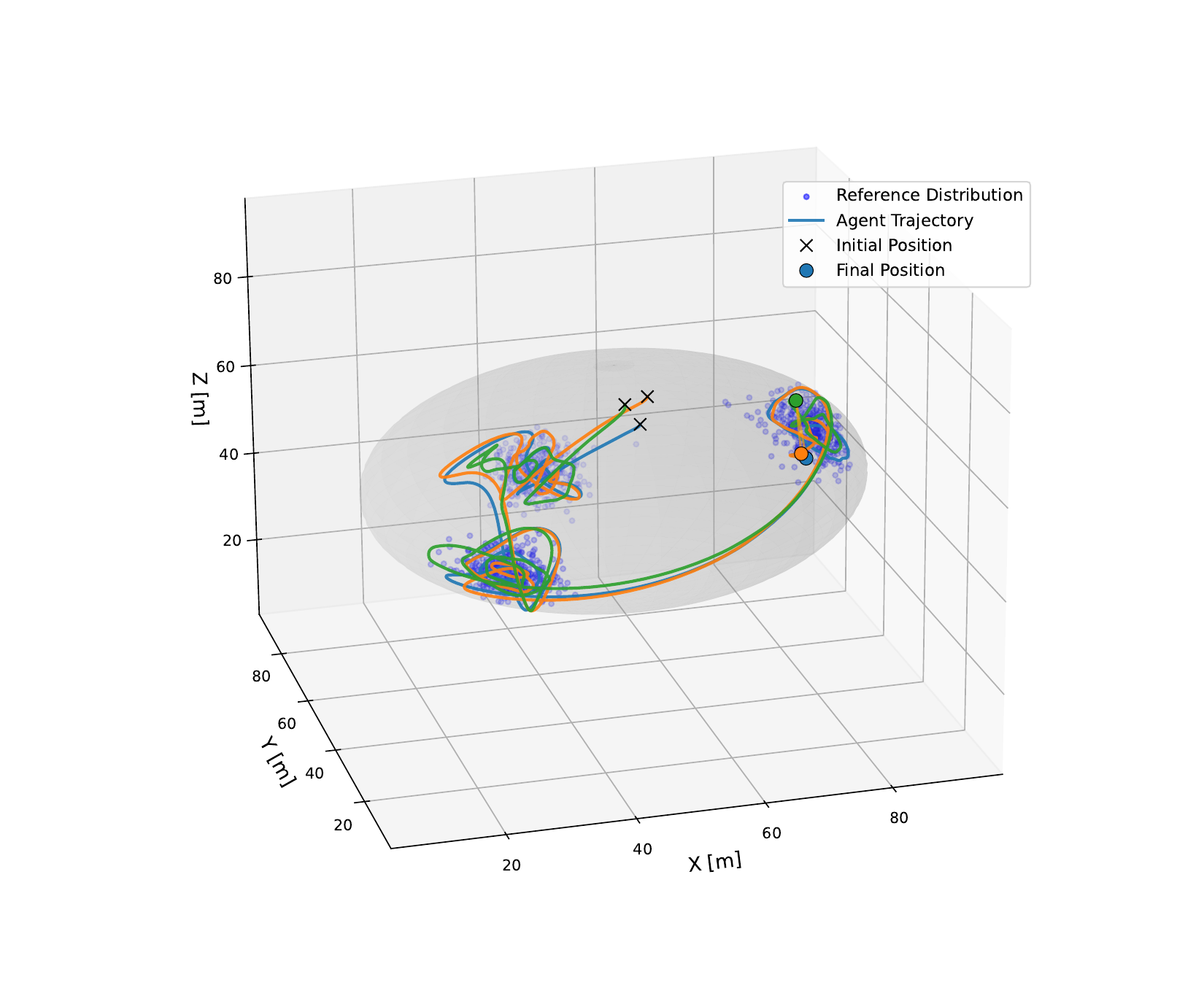}
    \vspace{-0.2in}
    \caption{$R$-D$^2$OC result: Three-dimensional multi-agent trajectories and reference distribution on the Riemannian ellipsoid manifold.}
    \label{fig:simulation_trajectory}
    \vspace{-.1in}
\end{figure}

\begin{figure}[!t]
    \centering
    \subfloat[]{
    \includegraphics[width=.48\linewidth]{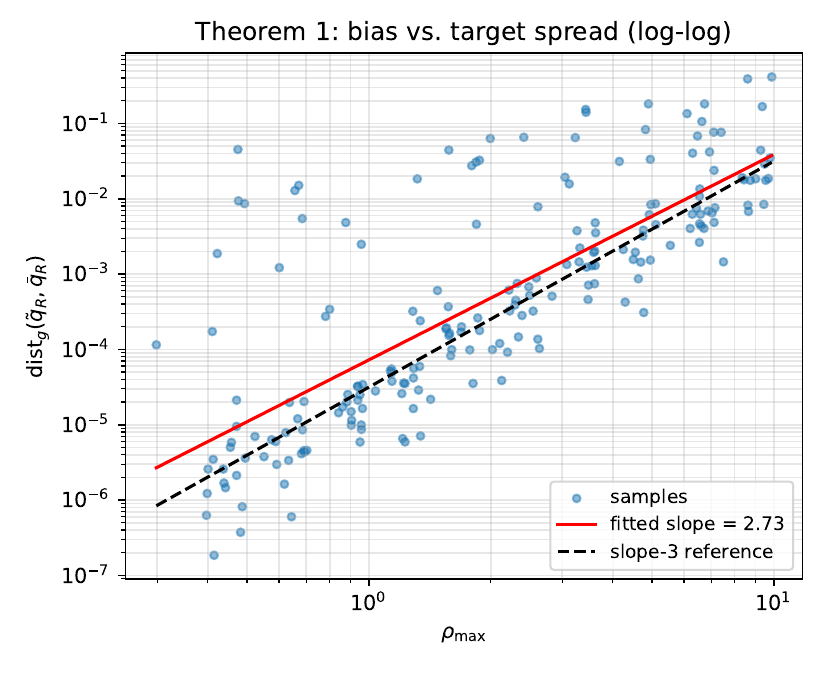}}
    \hfill
    \subfloat[]{
    \includegraphics[width=.48\linewidth]{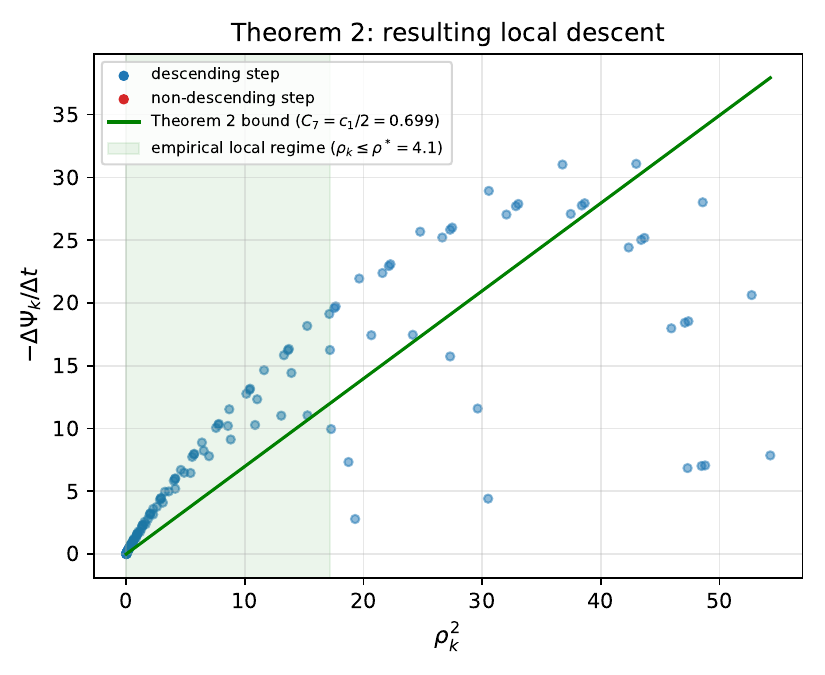}}
    \caption{Empirical verification of main theoretical results. (a) Approximation bias $\dist_g(\tilde{q}_R, \bar{q}_R)$ versus target spread, validating the third-order scaling in Theorem~1. (b) Empirical descent rate $-\Delta\Psi_k/\Delta t$ versus $\rho_k^2$, with the shaded region marking the neighborhood $\rho_k\le\rho^*\approx4.1$ where Theorem~\ref{thm:descent} holds.}
    \label{fig:descent_analysis}
    \vspace{-.2in}
\end{figure}

\section{Conclusion}

This paper presented an intrinsic $R$-D$^2$OC framework for second-order multi-agent systems on Riemannian manifolds. By mapping the intrinsic transport problem to local tangent spaces, the framework reduces online optimization to linear-quadratic subproblems with computational complexity depending strictly on the manifold dimension. We provided rigorous theoretical guarantees, including third-order barycentric bias bounds and a conditional discrete-time descent property. Numerical simulations on the sphere validated these theoretical results, confirming efficient distribution control on curved spaces.

\bibliographystyle{IEEEtran}
\bibliography{reference}

\end{document}